\documentclass{amsart}
\usepackage{amsmath}
\usepackage{amsfonts}
\usepackage{amssymb}
\usepackage{graphicx}%

\usepackage{color}  

\numberwithin{equation}{section}

\usepackage{url}
\def\R{\mathbb{R}}

\def\N{\mathbb{N}}
\def\NN{{\mathcal N}}
\def\cF{{\mathcal F}}
\def\vep{\varepsilon}
\def\bt{{\bf t}}
\newtheorem{theorem}{Theorem}

\newtheorem{lemma}[theorem]{Lemma}

\begin{document}


\title[Non-uniqueness for SPDE]
{Non-uniqueness for non-negative solutions of parabolic stochastic partial differential equations}

\author{K. Burdzy}
\address{Department of  Mathematics, U. of Washington, Seattle, USA}
\email{burdzy@math.washington.edu}

\author{C. Mueller}
\address{Department of Mathematics, U. of Rochester, Rochester, USA}
\email{http://www.math.rochester.edu/people/faculty/cmlr/}

\author {E.A. Perkins}
\address{Department of Mathematics, UBC, Vancouver, Canada}
\email{perkins@math.ubc.ca}

\thanks{Burdzy's research was supported in part by NSF Grant DMS-0906743 and
by grant N N201 397137, MNiSW, Poland.}
\thanks{Mueller's work was supported in part by  NSF Grant DMS-0705260.}
\thanks{Perkins' research was supported in part by an NSERC Discovery Grant.}
\date{\today}

\dedicatory { Dedicated to Don Burkholder.}

\begin{abstract}
Pathwise non-uniqueness is established for non-negative solutions of 
the parabolic stochastic pde 
$$\frac{\partial X}{\partial t}=\frac{\Delta}{2}X+X^p\dot W+\psi,\ X_0\equiv 0$$
where $\dot W$ is a white noise, $\psi\ge 0$ is smooth, compactly supported
and non-trivial, and $0<p<1/2$.  We further show that any solution spends
positive time at the $0$ function.

\end{abstract}

\maketitle
\section{Introduction}

Let $\sigma:\R\to\R$ be $p$-H\"older continuous (so 
$|\sigma(x)-\sigma(y)|\leq K|x-y|^p$), let $\psi\in C_c^1(\R)$ (the 
space of $C^1$ functions on $\R$ with compact support), and consider the 
parabolic stochastic partial differential equation
\begin{equation}\label{spde1}
\frac{\partial X}{\partial t}(t,x)=\frac{\Delta}{2}X(t,x)+\sigma(X(t,x))\dot W(t,x)+\psi.
\end{equation}
Here $\dot W$ is a space-time white noise on $\R_+\times \R$. If $\sigma$ is 
Lipschitz continuous, pathwise uniqueness of solutions to \eqref{spde1} is 
classical (see, e.g., \cite{Wal86}).  Particular cases of \eqref{spde1} for 
non-Lipschitz $\sigma$ arise in equations modeling populations undergoing 
migration (leading to the Laplacian) and critical reproduction or resampling 
(leading to the white noise term).  For example if $\sigma(X)=\sqrt X$ and 
$X\ge 0$, we have the equation for the density of one-dimensional 
super-Brownian motion with immigration $\psi$ (see Section III.4 of 
\cite{P01}). If $\sigma(X)=\sqrt{X(1-X)}$, $\psi=0$ and $X\in[0,1]$ we get the 
equation for the density of the stepping stone model on the line \cite{shi88}.
In both cases pathwise uniqueness of solutions remains 
open while uniqueness in law is obtained by (different) duality arguments (see 
the above references).  The duality arguments are highly non-robust and fail, 
for example if $\sigma(x,X)=\sqrt{f(x,X)X}$, which models a critically 
branching population with branching rate at site $x$ in state $X$ is $f(x,X)$.  
This is one  reason that there is interest in proving pathwise uniqueness in 
\eqref{spde1} under H\"older continuous conditions on $\sigma$, corresponding 
to the classical results of \cite{YW71} for one-dimensional SDE's with 
H\"older $1/2$-continuous diffusion coefficients. 

In \cite{MP10} pathwise uniqueness for \eqref{spde1} is proved if $p>3/4$ and 
in \cite{MMP11} pathwise uniqueness and uniqueness in law are shown to fail in 
\eqref{spde1} when $\sigma(X)=|X|^p$ for 
$1/2\le p<3/4$.  Here a non-zero solution to \eqref{spde1} is 
constructed for zero initial conditions and the signed nature of the solution 
is critical.  In the examples cited above the solutions of interest are 
non-negative and so it is natural to ask 
whether the results in \cite{MP10} 
can be improved if there is only one point (say $u=0$) where $\sigma(u)$ fails 
to be Lipschitz, and we are only interested in non-negative solutions. Finding 
weaker conditions which imply pathwise uniqueness of non-negative solutions in 
this setting is a topic of ongoing research. In this paper we give 
counterexamples to pathwise uniqueness of non-negative solutions in the 
admittedly easier setting where $p<1/2$.  Even here, however, we will find 
there are new issues which arise in our infinite dimensional setting. 
Our methods will also
allow us to extend the nonuniqueness result in \cite{MMP11} mentioned above to $0<p<1/2$.

We assume $\dot W$ is a white noise on the filtered probability space 
$(\Omega,\cF,\cF_t,P)$, where $\cF_t$ satisfies the usual hypotheses.  This 
means $W_t(\phi)$ is an $\cF_t$-Brownian motion with variance 
$\Vert\phi\Vert_2^2$ for each $\phi\in L^2(\R,dx)$ and $W_t(\phi_1)$ and 
$W_t(\phi_2)$ are independent if 
$\langle\phi_1,\phi_2\rangle\equiv\int\phi_1(x)\phi_2(x)dx=0$. A stochastic 
process $X:\Omega\times\R_+\times\R\to\R$ which is 
$\cF_t-\hbox{previsible}\times\hbox{Borel}$ measurable will be called a solution to the 
stochastic heat equation \eqref{spde1} with initial condition $X_0:\R\to\R$ if 
for each $\phi\in C_c^\infty(\R)$,
\begin{align*}
\langle X_t,\phi\rangle=&\langle X_0,\phi\rangle+\int_0^t\left\langle X_s,\frac{\Delta}{2}\phi\right\rangle ds\\
 &+\int_0^t\int \sigma(X(s,x))\phi(x)W(ds,dx)+t\langle\phi,\psi\rangle\hbox{ for all }t\ge 0 \ a.s.
\end{align*}
(The existence of all the integrals is of course part of the definition.) It 
is convenient to use the space $C_{rap}(\R)$ of rapidly decreasing continuous 
functions on $\R$ as a state space for our solutions.  To describe this 
space, for
$f\in C(\R)$ (the continuous functions on $\R$) let 
 \[|f|_\lambda=\sup_{x\in\R}e^{\lambda |x|}|f(x)|,\]
 and set
\begin{align*}
&C_{rap}=\{f\in C(\R):|f|_\lambda<\infty\ \forall \lambda>0\} 
,
  \\
&C_{tem}=\{f\in C(\R):|f|_\lambda<\infty\ \forall \lambda<0\}.
\end{align*}
Equip $C_{rap}$ with the complete metric
\[d(f,g)=\sum_{k=1}^\infty2^{-k}(\Vert f-g\Vert_k\wedge 1),\]
and $C_{tem}$ is given the complete metric
\[d_{tem}(f,g)=\sum_{k=1}^\infty 2^{-k}(\Vert f-g\Vert_{-1/k}\wedge 1).\]
Let $C^+_{rap}$ be the subspace of non-negative functions in $C_{rap}$, which 
is a Polish space. Our primary interest is in the smaller space 
$C^+_{rap}$ resulting in stronger non-uniqueness results.

A $C^+_{rap}$-valued solution to \eqref{spde1} is a solution $X$ such that $t\to X(t,\cdot)$ is in 
$C(\R_+,C^+_{rap})$, the space of continuous $C^+_{rap}$-valued paths for all $\omega$.  
In general if $E$ is a Polish space we give $C(\R_+,E)$ the topology of uniform convergence on compact sets.

The following result is proved just as in Theorem 2.5 of \cite{Shi94}.  

\begin{theorem}\label{existence} (Weak Existence of Solutions).  Assume $\psi\ge 0$ and the $p$-H\"older continuous function $\sigma$ satisfies $\sigma(0)=0$.  If $X_0\in C^+_{rap}$, there exists a filtered space $(\Omega,\cF,\cF_t,P)$ with a white noise $\dot W$ and a $C^+_{rap}$-valued solution of \eqref{spde1}.
\end{theorem}
\begin{proof} Our conditions on $\sigma$ imply the hypothesis on $a$ in 
Theorem~2.5 of \cite{Shi94}, however that reference assumes $\psi(x,X)$ 
satisfies $\psi(x,X)\le c|X|$.  The proof, however, extends easily to our 
simpler setting of $\psi(x)\ge 0$. \end{proof}

Here is our main result on non-uniqueness. The proof is given in Section~\ref{sec:proofnonpu}.  Recall that $\psi\in C^1_c(\R)$.

\begin{theorem}\label{thm:nonpu} Consider \eqref{spde1} with $\sigma(X)=|X|^p$ for $p\in(0,1/2)$ and $\psi\ge 0$ with $\int\psi(x)\,dx>0$.  There is a filtered space $(\Omega,\cF,\cF_t,P)$ carrying a white noise $\dot W$ and two $C^+_{rap}$-valued solutions to \eqref{spde1} with initial conditions $X_0^1=X^2_0=0$ such that $P(X^1\neq X^2)>0$.  That is, pathwise uniqueness fails for non-negative solutions to \eqref{spde1} for $\sigma$, $\psi$ as above. 
\end{theorem}

\noindent{\bf Remarks.} 1. The state of affairs in Theorem~\ref{thm:nonpu} for 
$\psi=0$ but $X_0$ non-zero remains unresolved. We expect the solutions to 
still be pathwise non-unique.  The methods used
to prove the above theorem do show pathwise uniqueness and uniqueness in law fail if $\psi=X_0=0$ and we drop the non-negativity condition on solutions. Namely, one can construct a non-zero solution to the resulting equation.  We will not prove this as stronger results (described above) will be shown
in \cite{MMP11} using different methods.  
\medskip

2. Uniqueness in law holds for non-negative solutions to \eqref{spde1} for $\psi$, $\sigma$ as above and general initial condition $X_0\in C^+_{rap}$ but now with $1\ge p\ge 1/2$. This may be proved as in \cite{My98} where the case $\psi=0$ is treated; for $p=1/2$ this is of course the well-known uniqueness of super-Brownian motion with immigration $\psi$.  We do not know if uniqueness in law fails for $p<1/2$.  The presence of a drift will play an important role in the proof of Theorem~\ref{thm:nonpu}.  

3. A key technique in this paper is to consider the total mass
$M_t = \langle X_t,1 \rangle$, and then apply Theorem 4 which, given the H\"older continuity of
$X(t,x)$, shows the brackets process $[M]$ to be bounded below by the integral of a power of $M$.
This in turn allows one to apply comparison arguments with one dimensional
diffusions.

\bigskip

In Section~\ref{sec:sde} below we prove that in the corresponding stochastic  
ordinary differential equation, although pathwise uniqueness again fails, 
uniqueness in law does hold. Of course the SDE is now one-dimensional so on 
one hand this is not surprising.  On the other hand, the manner in which 
uniqueness in law holds is a bit surprising as the SDE picks out a particular 
boundary behaviour which has the solution spending positive time at $0$ (see 
Section~\ref{sec:sde}).   This leads naturally to the following property for 
all solutions to the SPDE in Theorem~\ref{thm:nonpu}.

\begin{theorem}\label{sticky}
Assume $\sigma$ and $\psi$ are as in Theorem~\ref{thm:nonpu}. Let $X$ be any $C^+_{rap}$-valued solution to \eqref{spde1} with $X_0=0$. Then
\[\int_0^t1(X(s,x)\equiv0 \ \forall x)\,ds>0\hbox{ for all }t>0\ a.s.\]
\end{theorem}
The proof will be given in Section~\ref{sec:sticky} below. Let
\[b=\langle\psi,1\rangle>0.\]
We note that the above result fails for $p=1/2$ since in that case $Y_t=4\langle X_t,1\rangle$ is a Bessel squared process of parameter $4b$ satisfying an ordinary sde of the form
\[dY_t=2\sqrt{Y_t}dB_t+4bdt.\]
Such solutions spend zero time at $0$ (see for example, the analysis in Section V.48 of \cite{RW}.)

Finally we state the non-uniqueness result which complements that in \cite{MMP11} in the much easier regime of $p<1/2$. The solutions here will be signed.

\begin{theorem}\label{Girseg} If $0<p<1/2$ there is a $C_{rap}$-valued solution $X$ to
\begin{equation}\label{signedeq}
\frac{\partial X}{\partial t}(t,x)=\frac{\Delta X}{2}(t,x)+|X(t,x)|^p\dot W,\quad X(0)\equiv 0,
\end{equation}
so that $P(X\not\equiv 0)>0$. In particular uniqueness in law and pathwise uniqueness fail in 
\eqref{signedeq}.
\end{theorem}

Although the construction in \cite{MMP11} for $1/2\le p<3/4$ is more delicate, it is a bit awkward to extend
the reasoning to $p<1/2$ and so we prefer to present the result here.  The proof of Theorem~\ref{Girseg} is simpler than that of Theorem~\ref{thm:nonpu} in that we can focus on a single process rather than a pair of solutions.  
The two proofs are similar
in that approximate solutions are found by an excursion construction and the key 
ingredient required for the SPDE setting is Theorem~\ref{RAL} below.  Hence we only give a brief sketch of the proof of Theorem~\ref{Girseg}
 at the end of Section~\ref{sec:proofnonpu}.

\section{A Real Analysis Lemma}\label{sec:hollem}

\begin{theorem}\label{RAL} If $0<\alpha,\beta<1$ and $C>0$, there is a constant $K_{\ref{RAL}}(\beta,C)>0$ such that if $f:\R\to\R_+$ satisfies
\begin{equation}\label{holder} 
|f(x)-f(y)|\le C|x-y|^\beta,
\end{equation}
then
\begin{equation}
\nonumber\int f^\alpha \,dx\ge K_{\ref{RAL}}\Bigl(\int f\,dx\Bigr)^{(\alpha\beta+1)/(\beta+1)}.
\end{equation}
\end{theorem}
\begin{proof} 
First we use a scaling argument to reduce to the case
\[
\int f(x)dx = 1
\]
for which we would have to prove
\[
\int f^\alpha \geq K_4.
\]

Indeed, if we take $b>0$ and let
\[
g(x) = b^{-\beta} f(bx)
\]
then
\[
|g(x)-g(y)| = b^{-\beta}|f(bx)-f(by)| \leq C|x-y|^\beta
\]
by the conditions of Theorem \ref{RAL}.  Then setting $y=bx$, we get
\[
\int g(x)dx = \int b^{-\beta}f(bx)dx = b^{-(\beta+1)}\int f(y)dy = 1
\]
provided
\[
b = \left(\int f(y)dy\right)^{\frac{1}{\beta+1}}.  
\]
So $g$ satisfies the conditions of Theorem \ref{RAL} with $\int g=1$, and if 
we could show that 
\[
\int g^{\alpha}(x)dx \geq K_4
\]
it would follow, substituting for $b$, that 
\[
\int f^{\alpha}(x)dx = b^{\alpha\beta+1} \int g^\alpha(x)dx
\geq K_4 \left(\int f(x)dx\right)^{\frac{\alpha\beta+1}{\beta+1}}
\]
as required.  

Now we concentrate on proving $\int f^\alpha\geq K_4$ assuming that $\int f=1$ 
and assuming the H\"older condition \eqref{holder} on $f$.  Let $M=\sup_xf(x)$, and note the conclusion
is obvious if $M=\infty$ so assume it is finite.
If $M<1$, then since $0<\alpha<1$, we have
\[
\int f^{\alpha}(x)dx \geq \int f(x)dx = 1
.
\]
On the other hand, if $M\geq1$, then the H\"older condition on $f$ implies 
that $f\geq \frac{1}{2}$ on an interval $I$ whose length is bounded below by a 
constant $L>0$ depending only on $C,\beta$.  So in this case, too, we conclude
\[
\int f^{\alpha}(x)dx \geq \frac{L}{2^\alpha}\ge\frac{L}{2},
\]
and Theorem \ref{RAL} is proved.  
\end{proof}

\section{Proof of Theorem~\ref{thm:nonpu}}\label{sec:proofnonpu}

If $\psi\in C_c^1(\R)$, $\psi\ge 0$, $b=\int\psi dx>0$ and $0<p<1/2$, we want to construct distinct solutions $X,Y$ to
\begin{equation}\label{spde2}
\frac{\partial X}{\partial t}(t,x)=\frac{\Delta X}{2}(t,x)+(X(t,x))^p\dot W(t,x)+\psi(x),\ X\ge0, X_0=0.
\end{equation}
Let $C_b^k$ denote the space of bounded $C^k$ functions on $\R$ with bounded $j$th order partials for all $j\le k$, and set $C_b=C_b^0$.  The standard Brownian semigroup is denoted by $(P_t,t\ge 0)$ and $p_t(\cdot)$ is the Brownian density.

Here is an overview of the proof.  We will proceed by constructing approximate solutions $(X^\vep,Y^\vep)$ to \eqref{spde2} and then let $(X,Y)$ be an appropriate weak limit point of $(X^{\vep_n},Y^{\vep_n})$.  These approximate solutions will satisfy $X^\vep\ge Y^\vep\ge 0$ and $Y^\vep\ge X^\vep\ge 0$, respectively, on alternating excursions away from $0$ by $M=\langle X^\vep,1\rangle\vee\langle Y^\vep,1\rangle$. $M$ will equal $2b\vep$ at the effective start of each excursion.  We then calculate an upper bound on the probability that $M$ will hit $1$ on a given excursion (see \eqref{Qimaxhit} below) and a lower bound on 
$\bar D^\vep
  =|\langle X^\vep,1\rangle-\langle Y^\vep,1\rangle|$ hitting an appropriate $x_0\in (0,1)$ during each excursion (see \eqref{sepprob2} below).  Theorem~\ref{RAL} is used in the proof of the first bound (see \eqref{sqfn'} below).  These bounds will then show there is positive probability (independent of $\vep$) of 
$\bar D^\vep$ 
  hitting $x_0$ before $M$ hits $1$.  The result follows by taking weak limits as $\vep_n\downarrow 0$.  The use of Theorem~\ref{RAL} will mean the above upper bound is valid only up to a stopping time $V^\vep_k$ which will be large with high probability.  This necessitates a ``padding out" of the above excursions after this stopping time, and this technical step unfortunately complicates the construction.

Fix $\vep>0$ and define $(X^\vep,Y^\vep)$, $D^\vep=|X^\vep-Y^\vep|$, the white noise $\dot W$ and a sequence of stopping times inductively on $j$ as follows.  Let $T^\vep_0=0$, $U_j^\vep=T^\vep_j+\vep$, and assume $X^\vep_{T^\vep_{2j}}=Y^\vep_{T^\vep_{2j}}\equiv 0$ on $\{T^\vep_{2j}<\infty\}$.  Assuming $\{T^\vep_{2j}<\infty\}$, on $[T^\vep_{2j},U^\vep_{2j}]$ define 
\begin{equation}\label{UTdef1}
Y^\vep(t,x)\equiv0,\hbox{ and }D^\vep(t,\cdot)=X^\vep(t,\cdot)=2\int_0^{t-T^\vep_{2j}}P_s\psi(\cdot)ds\in C^+_{rap}.
\end{equation}
That is,
\begin{equation}\label{UTdef2}
\frac{\partial X^\vep}{\partial t}=\frac{\Delta}{2}X^\vep+2\psi\ \ \hbox{ for }T^\vep_{2j}\le t\le U^\vep_{2j}.
\end{equation}

Next let $t\to(Y^\vep_{U^\vep_{2j}+t},D^\vep_{U^\vep_{2j}+t})$ in $C(\R_+,C^+_{rap})^2$ solve the following SPDE for $t\ge U^\vep_{2j}$,
\begin{align}\label{YDdef}
\frac{\partial Y^\vep}{\partial t}&=\frac{\Delta}{2}Y^\vep+\psi+(Y^\vep)^p\dot W,\ Y^\vep_{U_{2j}^\vep}\equiv0
,
\\
\nonumber \frac{\partial D^\vep}{\partial t}&=\frac{\Delta}{2}D^\vep+[(Y^\vep+D^\vep)^p-(Y^\vep)^p]\dot W, D^\vep_{U^\vep_{2j}}(x)=2\int_0^\vep P_s\psi(x)ds.
\end{align}
The existence of such a solution on some filtered space carrying a white noise 
follows as in Theorem~2.5 of \cite{Shi94}.  To be 
careful here one has to construct an appropriate conditional probability given 
$\cF_{U^\vep_{2j}}$ and so inductively construct our white noise along with 
$(Y^\vep,D^\vep)$.  Set $X^\vep_t=Y^\vep_t+D^\vep_t$ for 
$U^\vep_{2j}\le t\le T^\vep_{2j+1}$, where
\[T^\vep_{2j+1}=\inf\{t\ge U^\vep_{2j}:\langle X^\vep_t,1\rangle=0\}\ (\inf\emptyset=\infty),\]
and also restrict the above definition of $(Y^\vep,D^\vep)$ to $[U^\vep_{2j},T^\vep_{2j+1}]$. 
Therefore on $[U^\vep_{2j},T^\vep_{2j+1}]$,
\begin{align}
\nonumber \frac{\partial Y^\vep}{\partial t}&=\frac{\Delta}{2}Y^\vep+\psi+(Y^\vep)^p\dot W,\ Y^\vep_{U^\vep_{2j}}\equiv0,\\
\label{XYdef}\frac{\partial X^\vep}{\partial t}&=\frac{\Delta}{2}X^\vep+\psi+(X^\vep)^p\dot W,\ X^\vep_{U^\vep_{2j}}(x)=2\int_0^\vep P_s\psi(x)ds,\\
\nonumber D^\vep&=|X^\vep-Y^\vep|=X^\vep-Y^\vep\ge 0,\\ 
\nonumber X^\vep,Y^\vep&\hbox{  are continuous and $C^+_{rap}$-valued}.
\end{align} 
Note that $X^\vep_{T^\vep_{2j+1}}=Y^\vep_{T^\vep_{2j+1}}=0$.  The precise meaning of the above 
formulas for $\partial Y^\vep/\partial t$ and $\partial X^\vep/\partial t$
 is that equality holds after multiplying by $\phi\in C_c^\infty$ and integrating over $\R$ and over any time interval in $[U^\vep_{2j},T^\vep_{2j+1}]$.

Now assume $T^\vep_{2j+1}<\infty$ and construct $(X^\vep,Y^\vep)$ and $D^\vep=|X^\vep-Y^\vep|=Y^\vep-X^\vep$ as above but with the roles of $X$ and $Y$ reversed. This means that on $[T^\vep_{2j+1},U^\vep_{2j+1}]=[T^\vep_{2j+1},T^\vep_{2j+1}+\vep]$,
\begin{equation}\label{UTdef4}D^\vep(t,\cdot)=Y^\vep(t,\cdot)=2\int_0^{t-T^\vep_{2j+1}}P_s\psi(\cdot)\,ds\in C^+_{rap}(\R)\hbox{ and }X^\vep(t,\cdot)\equiv0,\end{equation}
and so
\begin{equation}\label{Ydefep}
\frac{\partial Y^\vep}{\partial t}=\frac{\Delta}{2}Y^\vep+2\psi,
\end{equation}
and on $[U^\vep_{2j+1},T^\vep_{2j+2}]$,
\begin{align}
\nonumber\frac{\partial X^\vep}{\partial t}&=\frac{\Delta}{2}X^\vep+\psi+(X^\vep)^p\dot W,\ X^\vep_{U^\vep_{2j+1}}\equiv0,\ X^\vep\ge 0,\\
\label{UTdef3}\frac{\partial Y^\vep}{\partial t}&=\frac{\Delta}{2}Y^\vep+\psi+(Y^\vep)^p\dot W,\ Y^\vep_{U^\vep_{2j+1}}=2\int_0^\vep P_s\psi(x)\,ds,\ Y^\vep\ge 0,\\
\nonumber D^\vep&=|X^\vep-Y^\vep|=Y^\vep-X^\vep\ge 0\\
\nonumber X^\vep,Y^\vep&\hbox{  are continuous and $C^+_{rap}$-valued}.
\end{align} 
Here, as before, we have 
\[T^\vep_{2j+2}=\inf\{t\ge U^\vep_{2j+1}:\langle Y^\vep_t,1\rangle=0\}.\]

Clearly $X^\vep_{T^\vep_{2j+2}}=Y^\vep_{T^\vep_{2j+2}}\equiv0$ on $\{T^\vep_{2j+2}<\infty\}$
and $T^\vep_j\uparrow\infty$ and so our inductive construction of $(X^\vep,Y^\vep)$ is complete.
It is also clear from the construction that if $X_j(t)=X^\vep_{(T^\vep_j+t)\wedge T^\vep_{j+1}}$ and similarly for $Y_j$, then we may assume
\begin{align}
\label{exceven}
P((&X_{2j},Y_{2j},T^\vep_{2j+1}-T^\vep_{2j})\in\cdot|\cF_{T_{2j}^\vep})\\
\nonumber&=P((X_0,Y_0,T_1^\vep)\in\cdot)\ \hbox{ a.s. on }\{T^\vep_{2j}<\infty\},
\end{align}
and
\begin{align}
\label{excodd}
P((&Y_{2j+1},X_{2j+1},T^\vep_{2j+2}-T^\vep_{2j+1})\in\cdot|\cF_{T_{2j+1}^\vep})\\
\nonumber&=P((X_0,Y_0,T_1^\vep)\in\cdot)\ \hbox{ a.s. on }\{T^\vep_{2j+1}<\infty\}.
\end{align}

Define $J_\vep=
\bigcup
_{j=1}^\infty [U^\vep_{j-1},T^\vep_j]$, 
\[A_1^\vep(t,x)=\sum_{j=0}^\infty (-1)^j\int_0^t1_{[T_j^\vep,U^\vep_j]}(s)\psi(x)ds,\]
and $A_2^\vep(t,x)=-A_1^\vep(t,x)$. 
Combine \eqref{UTdef2}, \eqref{XYdef}, \eqref{UTdef3} and the fact that on $[T^\vep_{2j+1},U^\vep_{2j+1}]$ we have $X_t^\vep=0=\frac{\Delta}{2}X^\vep_t+\psi-\psi$ to see that for a test function $\phi\in C_c^\infty(\R)$,
\begin{align}\label{XMP}
\langle X_t^\vep,\phi\rangle=\langle A_1^\vep(t),\phi\rangle&+\int_0^t
\left (
\langle
X_s^\vep,\frac{\Delta}{2}\phi\rangle+\langle\psi,\phi\rangle
\right )
\, ds\\
\nonumber &+\int_0^t\int1_{J_\vep}(s)\phi(x)X^\vep(s,x)^pdW(s,x)
,
\\
\nonumber X^\vep_\cdot\in C(\R_+,C^+_{rap}).&
\end{align}
Similar reasoning gives 
\begin{align}\label{YMP}
\langle Y_t^\vep,\phi\rangle=\langle A_2^\vep(t),\phi\rangle&+\int_0^t
\left (
\langle
Y_s^\vep,\frac{\Delta}{2}\phi\rangle+\langle\psi,\phi\rangle
\right )
\, ds\\
\nonumber &+\int_0^t\int1_{J_\vep}(s)\phi(x)Y^\vep(s,x)^pdW(s,x)
,
\\
\nonumber Y^\vep_\cdot\in C(\R_+,C^+_{rap}).&
\end{align}

Since $U_j^\vep=T^\vep_j+\vep$, the alternating summation in the definition of $A^\vep_1$ implies that
\begin{equation}\label{Abnd} \sup_t|A_i^\vep(t,x)|\le \vep\psi(x).
\end{equation}
It follows from \eqref{UTdef1} and \eqref{UTdef4} (recall that $b=\int\psi(x)dx$) that
\[X^\vep(t,x)1_{J_\vep^c}(t)\le 2\int_0^\vep P_s\psi(x)\,ds\le4\vep^{1/2} b.\]
Therefore for any $T>0$ and $\phi$ as above
\begin{equation}\label{Martbnd}
E\Bigl(\Bigl[\int_0^T\int 1_{J_{\vep}^c}(s)(X^\vep(s,x))^p\phi(x)\,dW(s,x)\Bigr]^2\Bigr)\le(4\vep^{1/2}b)^{2p}T\Vert\phi\Vert_2^2.
\end{equation}

By identifying the white noise $\dot W$ with associated Brownian sheet, we may view $W$ as a stochastic process with sample paths in $C(\R_+,C_{tem}(\R))$.  Using bounds in Section 6 of \cite{Shi94} (see especially the $p$th moment bounds in the proofs of Theorems~2.2 and 2.5 there)
it is straightforward to verify that for $\vep_n\downarrow 0$, $\{(X^{\vep_n}, Y^{\vep_n},W):n\in\N\}$ is tight in $C(\R_+,(C_{rap}^+)^2\times C_{tem})$.  Some of the required bounds are in fact derived in the proof of Lemma~\ref{lem:unifHolder} below. By \eqref{Abnd}, \eqref{Martbnd} and their analogues for $Y^\vep$, one sees from \eqref{XMP} and \eqref{YMP} that for any limit point $(X,Y,W)$, $X$ and $Y$ are $C^+_{rap}$-valued solutions of \eqref{spde2} with respect to the common $\dot W$.  It remains to show that $X$ and $Y$ are distinct.

We know $X^\vep(t,\cdot)$ and $Y^\vep(t,\cdot)$ will be locally H\"older continuous of index $1/4$ 
but it will be convenient to have a slightly stronger statement. 
We note parenthetically that any other index of H\"older continuity for $X^\vep(t,\cdot)$ and $Y^\vep(t,\cdot)$ would yield the same range for $p$ in Theorem \ref{thm:nonpu}, provided that the index were
less than $1/2$.
 Let
\begin{align*}V_k^\vep=\inf\{s\ge 0:\ &\exists x,x'\in\R\hbox{ such that }\\
&|X^\vep(s,x)-X^\vep(s,x')|+|Y^\vep(s,x)-Y^\vep(s,x')|>k|x-x'|^{1/4}\}.\end{align*}
We will show in Lemma \ref{lem:unifHolder} that 
$\lim_{k\to\infty}\sup_{0<\vep\le 1}P(V^\vep_k\le M)=0$
for any $M\in\N$.

Note that on $[T^\vep_{2j},U^\vep_{2j}]$, $Y^\vep=0$ and 
\[|X^\vep(t,x)-X^\vep(t,x')|=2\Bigl|\int p_{t-T^\vep_{2j}}(z)(\psi(z+x)-\psi(z+x'))dz\Bigr|\le 2\Vert\psi'\Vert_\infty|x-x'|.\]
This implies that on the above interval for all real $x,x'$,
\[|X^\vep(t,x)-X^\vep(t,x')|+|Y^\vep(t,x)-Y^\vep(t,x')|\le 4(\Vert\psi'\Vert_\infty\vee\Vert\psi\Vert_\infty)|x-x'|^{1/4},\]
where the inequality holds trivially for $|x-x'|>1$ since the left side is at most $4\Vert\psi\Vert_\infty$.
By symmetry it also holds on $[T^\vep_{2j+1},U^\vep_{2j+1}]$.  We may assume\hfil\break $k\ge 4(\Vert\psi'\Vert_\infty\vee\Vert\psi\Vert_\infty)$ and so the above implies
\begin{equation}\label{Vkrest}
V^\vep_k\in
\bigcup
_{j=0}^\infty(U^\vep_j,T^\vep_{j+1}]\cup\{\infty\}.
\end{equation}

We fix a value of $k$ which will be chosen sufficiently large below.  We will now enlarge our probability space to include a pair of processes $(\bar X_t^\vep,\bar Y_t^\vep)$ which will equal 
$(\langle X_t^\vep,1\rangle,\langle Y_t^\vep,1\rangle)$ up to time $V_k^\vep$ and then switch to a pair of approximate solutions to a convenient SDE.  Set $p'=\frac{p+2}{5}\in(0,\frac{1}{2})$ and $K(k)=K_{\ref{RAL}}(1/4,k)$ where $K_{\ref{RAL}}$ is as in Theorem~\ref{RAL}. We may assume our $(\Omega,\cF,\cF_t,P)$ carries a standard $\cF_t$-Brownian motion $(B_s:s\le V^\vep_k)$, independent of $(X^\vep,Y^\vep,W)$. 

To define a law $Q_0$ on $C(\R,\R_+^2)\times[0,\infty]$, first construct 
a solution $(\tilde Y^\vep,
\tilde D^\vep,
 B)$ of 
\begin{align}\label{Q0sde}
\tilde Y^\vep_t&=b\int_0^t1(\vep<s)\,ds+\int_0^t1(\vep<s)(\tilde Y^\vep_s)^{p'}\sqrt{K(k)}\,dB_s,\ \tilde Y^\vep\ge 0
,
\\
\nonumber \tilde D^\vep_t&=2b(t\wedge\vep)+\int_0^t 1(\vep<s)[(\tilde Y^\vep_s+\tilde D_s^\vep)^{p'}-(\tilde Y^\vep_s)^{p'}]\sqrt{K(k)}\,dB_s,\ \tilde D^\vep\ge 0.
\end{align}
Such a weak solution may again be found by approximation by solutions of Lipschitz SDE's as in Theorems~2.5 and 2.6 of \cite{Shi94} for the more complicated stochastic pde setting.  Set $\tilde X^\vep=\tilde Y^\vep+\tilde D^\vep$ and $\tilde T^\vep_1=\inf\{t:\tilde X^\vep_t=0\}\ge \vep$, and define
\begin{align}\label{def:Q0}
Q_0(A)=P((\tilde X^\vep_{\cdot\wedge\tilde T^\vep_1},\tilde Y^\vep_{\cdot\wedge\tilde T^\vep_1},\tilde T^\vep_1)\in A).
\end{align}

Next we enlarge our space to include $(\bar X^\vep,\bar Y^\vep)$ so that for finite $t\le \bar T^\vep_1$ (this time is defined below),
\begin{align}
\label{barYdef}\bar Y^\vep_t&=\langle Y^\vep_{t\wedge V^\vep_k},1\rangle+ b(t-V^\vep_k)^+\\
\nonumber&\phantom{=\langle Y^\vep_{t\wedge V^\vep_k},1\rangle}+\int_0^t1(s>V_k^\vep)(\bar Y^\vep_s)^{p'}\sqrt{K(k)}\,dB_s,\ \bar Y^\vep\ge 0,\\
\label{barDdef}\bar D^\vep_t&=\langle D^\vep_{t\wedge V^\vep_k},1\rangle\\
\nonumber&\quad+\int_0^t1(s>V^\vep_k)[(\bar D^\vep_s+\bar Y^\vep_s)^{p'}-(\bar Y^\vep_s)^{p'}]\sqrt{K(k)}\,dB_s,\ \bar D^\vep\ge 0,\\
\label{barXdef}\bar X^\vep_t&=\bar Y^\vep_t+\bar D^\vep_t=\langle X^\vep_{t\wedge V^\vep_k},1\rangle+ b(t-V^\vep_k)^+\\
\nonumber&\phantom{=\bar Y^\vep_t+\bar D^\vep_t=\langle X^\vep_{t\wedge V^\vep_k},1\rangle}+\int_0^t1(s>V_k^\vep)(\bar X^\vep_s)^{p'}\sqrt{K(k)}\,dB_s,\\
\label{barTdef}\bar T^\vep_1&=\inf\{t\ge0:\bar X^\vep_t=0\}\le \infty.
\end{align}
Note that if $V_k^\vep\ge T_1^\vep$, then $\bar X_t^\vep=\langle X_t^\vep,1\rangle$ for $t\le T^\vep_1$ and so $\bar T^\vep_1=T^\vep_1$. Therefore, $\bar T^\vep_1\wedge V_k^\vep\le T_1^\vep$. We conclude 
that $\bar X^\vep_{t\wedge V^\vep_k}=\bar Y^\vep_{t\wedge V^\vep_k}+\bar D^\vep_{t\wedge V^\vep_k}$ for $t\le \bar T^\vep_1$, thus proving \eqref{barXdef}.

To carry out the above construction first build 
$(\bar Y^\vep_{V_k^\vep+t},\bar D^\vep_{V^\vep_k+t}, B_{V_k^\vep+t}-B_{V_k^\vep})$
by approximation by solutions to SDE's with Lipschitz coefficients as in 
Theorem~2.5 of \cite{Shi94}.  This and a measurable selection argument (see Section 12.2 of 
\cite{SV}) 
allows us 
to build the appropriate regular conditional probability
\begin{align*}Q^0&_{\langle Y^\vep_{V^\vep_k},1\rangle,\langle D^\vep_{V^\vep_k},1\rangle}(\cdot)\\
&\equiv P(((\bar Y^\vep_{V_k^\vep+t},\bar D^\vep_{V_k^\vep+t}, B(V_k^\vep+t)-B(V^\vep_k)), t\ge 0)\in\cdot|X^\vep,Y^\vep,W),
\end{align*}
where $\{Q^0_{y,d}:
y,d\ge 0\}$ 
 is a measurable family of laws on $C(\R_+,\R_+^2\times\R)$.  
This then allows us to construct $(\bar X^\vep, \bar Y^\vep,\bar D^\vep)$ as above on an enlargement of our original space which we still denote $(\Omega,\cF,\cF_t,P)$.  We also may now prescribe another measurable family of laws $\{Q_{y,x}:(y,x)\in C(\R_+,\R_+^2)\}$ on $C(\R_+,\R_+^2)\times [0,\infty]$ such that for each Borel $A$, w.p. 1,
\begin{align}
\label{Qdefn}Q_{\langle X^\vep_{\cdot\wedge V_k^\vep},1\rangle,\langle Y^\vep_{\cdot\wedge V_k^\vep},1\rangle}(A)=P((\bar X^\vep_{\cdot\wedge\bar T_1^\vep},\bar Y^\vep_{\cdot\wedge\bar T_1^\vep},\bar T^\vep_1)\in A |X^\vep, Y^\vep, W).
\end{align}
Define
\begin{align}\label{Q1defn}
Q_1(A)&=P((\bar X^\vep_{\cdot\wedge\bar T_1^\vep},\bar Y^\vep_{\cdot\wedge\bar T_1^\vep},\bar T^\vep_1)\in A)=E\Bigl(Q_{\langle X^\vep_{\cdot\wedge V_k^\vep},1\rangle,\langle Y^\vep_{\cdot\wedge V_k^\vep},1\rangle}(A) \Bigr).
\end{align}

Next, inductively define $(\bar X^\vep_t,\bar Y^\vep_t)$, $t\in[\bar T^\vep_j,\bar T^\vep_{j+1}]$, and $\{\bar T^\vep_j\}$ in a manner reminiscent of that for $(X^\vep,Y^\vep)$, and consistent with the above construction for $j=0$ (set $\bar T^\vep_0=0$). Assume the construction up to $\bar T^\vep_{2j}$ is such that 
\begin{equation}\label{barX0}
\bar X^\vep_{\bar T^\vep_{2j}}=\bar Y^\vep_{\bar T^\vep_{2j}}=0\hbox{ on }\{\bar T^\vep_{2j}<\infty\},
\end{equation}
and
\begin{equation}\label{TbarT}
\bar T^\vep_{2j}=T^\vep_{2j}\hbox{ on }\{V_k^\vep>\bar T^\vep_{2j}\}.
\end{equation}
Define
\begin{equation}\label{barXjdefn}
\bar X_j(t)=\bar X^\vep_{(\bar T^\vep_j+t)\wedge \bar T^\vep_{j+1}},
\end{equation}
and similarly define $\bar Y_j$. 
On $\{\bar T_{2j}^\vep<V_k^\vep\}$ set
\begin{align}\label{barT2j1}
P(&(\bar X_{2j},\bar Y_{2j},\bar T^\vep_{2j+1}-\bar T_{2j}^\vep)\in\cdot|\cF_{\bar T^\vep_{2j}}\vee\sigma(X^\vep,Y^\vep,W))\\
\nonumber&=Q_{\langle X^\vep_{(T^\vep_{2j}+\cdot)\wedge V_k^\vep},1\rangle,\langle Y^\vep_{(T^\vep_{2j}+\cdot)\wedge V_k^\vep},1\rangle}(\cdot).
\end{align}
On $\{\infty>\bar T^\vep_{2j}\ge V^\vep_k\}$ set
\begin{equation}\label{barT2j2}
P((\bar X_{2j},\bar Y_{2j},\bar T^\vep_{2j+1}-\bar T_{2j}^\vep)\in\cdot|\cF_{\bar T^\vep_{2j}}\vee\sigma(X^\vep,Y^\vep,W))=Q_0(\cdot).
\end{equation}
These definitions imply $\bar T^\vep_{2j+1}=\inf\{t>\bar T^\vep_{2j}:\bar X^\vep_t=0\}$ and that on our enlarged probability space, conditional on $\cF_{\bar T^\vep_{2j}}$ and on $\{V_k^\vep>\bar T_{2j}^\vep\}$, \eqref{barYdef}-\eqref{barXdef} hold for $t\in[\bar T^\vep_{2j},\bar T^\vep_{2j+1}]$, while on $\{\infty>\bar T_{2j}^\vep\ge V_k^\vep\}$, for $t\in[\bar T^\vep_{2j},\bar T^\vep_{2j+1}]$, 
\eqref{Q0sde}, \eqref{def:Q0} and \eqref{barT2j2} give
\begin{align}
\label{barT2j3}
\nonumber\bar Y^\vep_t&=b\int_0^t1(\bar T^\vep_{2j}+\vep<s)\,ds\\
\nonumber&\quad+\int_0^t1(\bar T^\vep_{2j}+\vep<s)(\bar Y^\vep_s)^{p'}\sqrt{K(k)}\,dB(s),\ \bar Y^\vep\ge 0,\\
\bar D^\vep_t&=2b\bigl((t-\bar T_{2j}^\vep)\wedge\vep\bigr)\\
\nonumber&\quad+\int_0^t1(\bar T^\vep_{2j}+\vep<s)\bigl((\bar Y^\vep_s+\bar D^\vep_s)^{p'}-(\bar Y^\vep_s)^{p'}\bigr)\sqrt{K(k)}\,dB_s,\ \bar D^\vep\ge 0,\\
\nonumber \bar X^\vep_t&=2b\bigl((t-\bar T^\vep_{2j})\wedge\vep\bigr)+b\int_0^t1(\bar T^\vep_{2j}+\vep<s)\,ds\\
\nonumber&\phantom{=2b\bigl((t-\bar T^\vep_{2j})\wedge\vep\bigr)}+\int_0^t1(\bar T^\vep_{2j}+\vep<s)(\bar X^\vep_s)^{p'}\sqrt{K(k)}\,dB(s).
\end{align}

Now assume $\bar T^\vep_{2j+1}<\infty$ and construct $(\bar X^\vep_t,\bar Y^\vep_t)$ and $\bar D^\vep_t=|\bar X^\vep_t-\bar Y^\vep_t|$ for $t\in[\bar T^\vep_{2j+1},\bar  T^\vep_{2j+2}]$ as above but with the roles of $\bar X$ and $\bar Y$ reversed. This means that on $\{\bar T^\vep_{2j+1}<V_k^\vep\}$,
\begin{align}\label{barT2j4} P(&(\bar Y_{2j+1},\bar X_{2j+1},\bar T^\vep_{2j+2}-\bar T^\vep_{2j+1})\in\cdot|\cF_{\bar T^\vep_{2j+1}}\vee\sigma(X^\vep,Y^\vep,W))\\
\nonumber&\qquad\qquad=Q_{\langle Y^\vep_{(T^\vep_{2j+1}+\cdot)\wedge V_k^\vep},1\rangle,\langle X^\vep_{(T^\vep_{2j+1}+\cdot)\wedge V_k^\vep},1\rangle}(\cdot),
\end{align}
and on $\{\infty>\bar T^\vep_{2j+1}\ge V_k^\vep\}$, the above conditional probability is again $Q_0$.  The apparent lack of symmetry in the definitions arises because we have also reversed the roles of $X^\vep$ and $Y^\vep$ on $[\bar T^\vep_{2j+1},\bar  T^\vep_{2j+2}]$. 
The above definition implies that $\bar D^\vep_t=\bar Y^\vep_t-\bar X^\vep_t\ge 0$ on $[\bar T^\vep_{2j+1},\bar  T^\vep_{2j+2}]$, $\bar T^\vep_{2j+2}=\inf\{t>\bar T^\vep_{2j+1}:\bar Y^\vep_t=0\}$, and $\bar X^\vep(\bar T^\vep_{2j+2})=\bar Y^\vep(\bar T^\vep_{2j+2})=0$ on $\{\bar T^\vep_{2j+2}<\infty\}$. 

It follows from \eqref{barYdef}, \eqref{barXdef} (now with $t\in [\bar T^\vep_{2j},\bar T^\vep_{2j+1}]$) and \eqref{TbarT} that on $\{V^\vep_k>\bar T^\vep_{2j+1}\}$ we have $\bar T^\vep_{2j+1}=T^\vep_{2j+1}$. Symmetric reasoning shows that on $\{V^\vep_k>\bar T^\vep_{2j+2}\}$, $\bar T^\vep_{2j+2}= T^\vep_{2j+2}$.  We have verified \eqref{barX0} and \eqref{TbarT} for $j+1$. Since $\bar T^\vep_{j+1}-\bar T^\vep_j\ge \vep$ (by \eqref{barYdef},\eqref{barXdef} and \eqref{barT2j3}), $\bar T^\vep_j\uparrow \infty$ and our inductive definition is complete.

The reasoning above to show $\bar T^\vep_{2j+1}= T^\vep_{2j+1}$ on $\{V_k^\vep>\bar T^\vep_{2j+1}\}$ and the obvious induction also shows that 
\begin{align}\label{barXeqX} \bar X^\vep_{t\wedge V^\vep_k}&=\langle X^\vep_{t\wedge V_k^\vep},1\rangle, \ \bar Y^\vep_{t\wedge V^\vep_k}=\langle Y^\vep_{t\wedge V_k^\vep},1\rangle,\\
\nonumber \bar D^\vep_{t\wedge V_k^\vep}&=|\langle X^\vep_{t\wedge V_k^\vep},1\rangle-\langle Y^\vep_{t\wedge V_k^\vep},1\rangle|\quad\forall t\ge 0\hbox{ a.s.}
\end{align}

The following consequence of the above construction will be important for us:
\begin{align}\label{barexc1}
\nonumber P\bigl((&\bar X_{2j},\bar Y_{2j},\bar T^\vep_{2j+1}-\bar T^\vep_{2j})\in\cdot\bigl|\cF_{\bar T^\vep_{2j}}\bigr)\\
&\qquad\qquad=Q_1(\cdot)\hbox{ a.s. on }\{\bar T^\vep_{2j}<V_k^\vep\},\\
\nonumber P\bigl((&\bar Y_{2j+1},\bar X_{2j+1},\bar T^\vep_{2j+2}-\bar T_{2j+1})\in\cdot\bigl|\cF_{\bar T^\vep_{2j+1}}\bigr)\\
\nonumber&\qquad\qquad=Q_1(\cdot)\hbox{ a.s. on }\{\bar T^\vep_{2j+1}<V_k^\vep\},
\end{align}
and
\begin{align}\label{barexc2}
\nonumber P\bigl((&\bar X_{2j},\bar Y_{2j}, \bar T^\vep_{2j+1}-\bar T^\vep_{2j})\in\cdot\bigl|\cF_{\bar T^\vep_{2j}}\bigr)\\
&\qquad\qquad=Q_0(\cdot)\hbox{ a.s. on }\{V^\vep_k\le \bar T^\vep_{2j}\},\\
\nonumber P\bigl((&\bar Y_{2j+1},\bar X_{2j+1}, \bar T^\vep_{2j+2}-\bar T^\vep_{2j+1})\in\cdot\bigl|\cF_{\bar T^\vep_{2j+1}}\bigr)\\
\nonumber&\qquad\qquad=Q_0(\cdot)\hbox{ a.s. on }\{V^\vep_k\le \bar T^\vep_{2j+1}\}.
\end{align}

Consider, for example, the first equality in \eqref{barexc1}.  By \eqref{barT2j1}
 we have for a Borel set $B$ and $A\in\cF_{\bar T^\vep_{2j}}$, $A\subset\{V_k^\vep>\bar T^\vep_{2j}\}$,
 \begin{align}
 \nonumber P\bigl(&\{(\bar X_{2j},\bar Y_{2j},\bar T^\vep_{2j+1}-\bar T^\vep_{2j})\in B\}\cap A\bigr)\\
\nonumber &=E\Bigl(Q_{\langle X^\vep_{(T^\vep_{2j}+\cdot)\wedge V^\vep_k},1\rangle,\langle Y^\vep_{(T^\vep_{2j}+\cdot)\wedge V^\vep_k},1\rangle}(B)1_A\Bigr)\\
 \label{Exc1}&=E\Bigl(E\Bigl(Q_{\langle X^\vep_{(T^\vep_{2j}+\cdot)\wedge V^\vep_k},1\rangle,\langle Y^\vep_{(T^\vep_{2j}+\cdot)\wedge V^\vep_k},1\rangle}(B)\Bigl|\cF_{T^\vep_{2j}\wedge V^\vep_k}\Bigr)1_A\Bigr).
 \end{align}
In the last line we used the fact that $V^\vep_k>\bar T^\vep_{2j}=T^\vep_{2j}$ 
on $A$ to see that $A\in\cF_{T^\vep_{2j}\wedge V^\vep_k}$.  Formula \eqref{barXeqX} shows that our construction of $(\bar X^\vep,\bar Y^\vep)$ 
has not increased the information in $\cF_{T^\vep_{2j}\wedge V^\vep_k}$
so we may use \eqref{exceven}. 
Applying \eqref{exceven} and the 
fact that $V_k^\vep=T^\vep_{2j}+V_k^\vep\circ\theta_{T^\vep_{2j}}$ on 
$\{V_k^\vep>T_{2j}^\vep\}$, where $(\theta_t)$ are the shift operators for 
$(X^\vep,Y^\vep)$, we conclude from \eqref{Exc1} that the far left-hand side 
of \eqref{Exc1} equals
\[E\Bigl(Q_{\langle X^\vep_{\cdot\wedge V^\vep_k},1\rangle,\langle Y^\vep_{\cdot\wedge V_k^\vep},1\rangle}(B)\Bigr)1_A)=E\Bigl(Q_1(B)1_A\Bigr),\]
by \eqref{Q1defn}. This gives the first equality in \eqref{barexc1} and the second inequality holds by a symmetric argument.  The proof of \eqref{barexc2} is easier.

Our next goal is to show there is positive probability, independent of $\vep$, 
of $\bar D^\vep$ hitting some appropriately chosen $x_0\in(0,1)$ before 
$\bar X^\vep$ or $\bar Y^\vep$ hits $1$.  By \eqref{barexc1} and 
\eqref{barexc2} the excursions of $\bar X^\vep\vee\bar Y^\vep$ away 
from
 $0$ 
are governed by $Q_0$ or $Q_1$, depending on whether or not $V_k^\vep$ has 
occurred. Therefore we need to analyze these two laws.  

Consider the more complex $Q_1$ first. Use \eqref{XMP}, with $\phi=1$, in 
\eqref{barXdef} and the fact that $V_k^\vep>\vep$ (by \eqref{Vkrest}) to 
conclude that under $Q_1$, $\bar X^\vep_t=2b(t\wedge \vep)$ for $t\le \vep$ 
and for $0\le t\le \bar T^\vep_1-\vep$,
\begin{align}
\nonumber\bar X^\vep_{t+\vep}&=\vep b+((t+\vep)\wedge V_k^\vep)b+b\int_0^{t+\vep}
1(s>V_k^\vep)\,ds\\
\nonumber&\phantom{=\vep b}+\int_\vep^{t+\vep}\int 1(s\le V_k^\vep) X^\vep(s,x)^p\,dW(s,x)\\
\nonumber&\phantom{=\vep b}+\int_\vep^{t+\vep}1(s>V_k^\vep)(\bar X^\vep_s)^{p'}\sqrt{K(k)}\,dB_s\\
\label{barXMP}&=2\vep b+tb+N_t,
\end{align}
where $N$ is a continuous $(\cF_{t+\vep})-$local martingale such that 
\begin{align*}
\langle N\rangle_t&=\int_\vep^{t+\vep}\Bigl[1(s\le V^\vep_k)\int X^\vep(s,x)^{2p}dx+1(s>V_k^\vep)(\bar X^\vep_s)^{2p'}K(k)\Bigr]\,ds\\
&\equiv \int_\vep^{t+\vep}\langle N\rangle'(s)\,ds.
\end{align*}
By the definition of $V^\vep_k$ we may apply Theorem~\ref{RAL} with $(\alpha,\beta,C)=(2p,1/4,k)$ and conclude that 
\begin{align}\label{sqfn'}
\langle N\rangle'(s)&\ge 1(s\le V^\vep_k)K(k)\Bigl[\int X^\vep(s,x)\,dx\Bigr]^{((p/2)+1)/(5/4)}\\
\nonumber&\phantom{\ge A}+1(s>V_k^\vep)(\bar X^\vep_s)^{2p'}K(k)\\
\nonumber&=K(k)(\bar X^\vep_s)^{2p'},
\end{align}
where \eqref{barXeqX} is used in the last line.

Define a random time change $\tau_t$ by
\begin{equation}\label{taudef}
t=\int_0^{\tau_t}\frac{\langle N\rangle'(s+\vep)}{K(k)(\bar X^\vep_{s+\vep})^{2p'}}\,ds\equiv A(\tau_t),\ \ t<A(\bar T_1^\vep-\vep).
\end{equation}
The restriction on $t$ ensures we are not dividing by zero in the above integrand because $\bar T^\vep_1$ is the hitting time of $0$ by $\bar X^\vep$.  Clearly \eqref{sqfn'} implies 
\begin{equation}\label{tau'bnd}
\tau'(t)\le 1\hbox{ for }t<A(\bar T^\vep_1-\vep).
\end{equation}
For $t<A(\bar T_1^\vep-\vep)$, let
\begin{equation}\label{Xhatdef}
\hat X(t)=\bar X^\vep(\tau_t+\vep)=2b\vep+b\tau(t)+\hat N(t),
\end{equation}
where $\hat N_t=N(\tau_t)$ is continuous $(\cF_{\tau_t+\vep})$-local martingale such that
\begin{equation}\nonumber\langle\hat N\rangle_t=\int_\vep^{\tau_t+\vep}\langle N\rangle'(s)\,ds=\int_0^{\tau_t}\langle N\rangle'_{s+\vep}\,ds=K(k)\int_0^t \hat X(r)^{2p'}\,dr.\end{equation}
This follows by using the substitution $s=\tau_r$ and calculating the differential $d\tau(r)$ from \eqref{taudef}. Note also that if $\hat T_x=\inf\{t\ge 0:\hat X(t)=x\}$, then $A(\bar T_1^\vep-\vep)=\hat T_0$. 
Therefore by \eqref{Xhatdef} we may assume there is a Brownian motion $\hat B$ so that
\begin{equation}\label{hatXsde}
\hat X(t\wedge\hat T_0)=2b\vep+b\tau(t\wedge \hat T_0)+\int_0^{t\wedge \hat T_0}\sqrt{K(k)}\hat X(s)^{p'}d\hat B(s).
\end{equation}
The scale function for a  diffusion defined by a similar formula, but with $t\land \hat T_0$ in place of $\tau(t\land \hat T_0)$,
is 
\[s_k(x)=\int_0^x\exp\Bigl\{-\frac{2by^{1-2p'}}{K(k)(1-2p')}\Bigr\}\,dy.\]
That is, $s_k$ satisfies 
\begin{equation}\label{seq}\frac{K(k)x^{2p'}}{2}s_k''(x)+bs_k'(x)=0\hbox{ on }[0,\infty),\ s_k(0)=0.
\end{equation}
By It\^o's Lemma
\begin{align*}
s_k(\hat X(t\wedge \hat T_0))=s_k(2b\vep)&+\int_0^{t\wedge\hat T_0}b\tau'(u)s_k'(\hat X(u))+\frac{K(k)\hat X(u)^{
2p'
}}{2}s_k''(\hat X(u))\,du\\
&+\int_0^{t\wedge\hat T_0}s_k'(\hat X(u))d\hat N(u).
\end{align*}
\eqref{tau'bnd} and \eqref{seq} show that the integrand in the drift term above is non-positive, and so $s_k(\hat X(t\wedge \hat T_0\wedge \hat T_1))$ is a supermartingale which therefore satisfies
\[E\bigl(s_k(\hat X(t\wedge \hat T _0\wedge \hat T_1)\bigr)\le s_k(2b\vep).\]
This implies that
\begin{align}\nonumber Q_1(\bar X^\vep_t
=1 \hbox{  for some }
t<\bar T_1^\vep)&=Q_1(s_k(\hat X(\cdot\wedge \hat T_1\wedge\hat T_0))\hbox{ hits }s_k(1)\hbox{ before }0)\\
\nonumber&
=
\lim_{t\to\infty}Q_1\Bigl(s_k(\hat X (t\wedge \hat  T_0\wedge\hat T_1))/s_k(1)\Bigr)\\
\label{Q1hit}&\le s_k(2b\vep)/s_k(1).
\end{align}

Under $Q_0$ add the equations in \eqref{Q0sde} to see that (we write 
$\bar X^\vep$ for $\tilde X^\vep$),
\[\bar X^\vep_{t+\vep}=2b\vep+bt+\int_\vep^{t+\vep}(\bar X^\vep_s)^{p'}\sqrt{K(k)}\,dB_s,\ t+\vep\le \bar T^\vep_1=\inf\{t:\bar X^\vep_t=0\}.\]
This is equation \eqref{hatXsde} with $t$ in place of $\tau_t$ and so the previous calculation applies to again give us \eqref{Q1hit} with $Q_0$ in place of $Q_1$.  

Under either $Q_i$, $\bar X^\vep_t=\bar X^\vep_t\vee \bar Y^\vep_t$ and so we conclude
\begin{equation}\label{Qimaxhit}
Q_i(\bar X^\vep_t\vee\bar Y^\vep_t\hbox{ hits $1$ for }t<\bar T^\vep_1)\le s_k(2b\vep)/s_k(1),\ i=1,2.
\end{equation}

We next consider the escape probability for $\bar D^\vep$ under $Q_1$.  Let $x_0\in(2b\vep,1)$ and 
\[T_{\bar D}(0,x_0)=\inf\{t:\bar D^\vep_t=0\hbox{ or }x_0\}\le \bar T^\vep_1\ \ Q_1-a.s.,\]
the last since $\bar D^\vep_t=\bar X^\vep_t-\bar Y^\vep_t\le \bar X^\vep_t$ for $t\le \bar T^\vep_1$ under $Q_1$.  It follows from \eqref{UTdef1}, \eqref{YDdef} and \eqref{barDdef} that $\bar D_t^\vep=2b(t\wedge\vep)$ for $t\le \vep$, and for $t+\vep\le\bar T_1^\vep$ we have,
\begin{align}\label{barDsde}
\bar D^\vep_{t+\vep}=2b\vep&+\int_\vep^{t+\vep}\int(X^\vep(s,x)^p-Y^\vep(s,x)^p)1(s\le V_k^\vep)\,dW(s,x)\\
\nonumber &+\int_\vep^{t+\vep}1(s>V_k^\vep)((\bar X^\vep_s)^{p'}-(\bar Y^\vep_s)^{p'})\sqrt{K(k)}\,dB_s,
\end{align}
which is a non-negative local martingale in $t$. We have 
\begin{align}\label{Q1Dhit}
\nonumber Q_1&
(\exists t<\bar T^\vep_1: \bar D_t^\vep\ge x_0)
\\
\nonumber&\ge E\Bigl(\bar D^\vep(T_{\bar D}(0,x_0))x_0^{-1}1(T_{\bar D}(0,x_0)<\infty, \bar T_1^\vep<\infty)\Bigr)\\
&=E\Bigl(\bar D^\vep(T_{\bar D}(0,x_0)\wedge \bar T^\vep_1)x_0^{-1}\Bigr)-E\Bigl(\bar D^\vep(T_{\bar D}(0,x_0))x_0^{-1}1(\bar T_1^\vep=\infty)\Bigr).
\end{align}
The first term on the right-hand side is the terminal element of a bounded martingale and so
\begin{equation}\label{term1}
E\Bigl(\bar D^\vep(T_{\bar D}(0,x_0)\wedge \bar T^\vep_1)x_0^{-1}\Bigr)=2b\vep/x_0.
\end{equation}

It follows from \eqref{barXMP} that on $\{\bar T_1^\vep=\infty\}$,
\begin{equation}\label{DS}\{\limsup_{t\to\infty}\bar X^\vep_t<\infty\}\subset\{\lim_{t\to\infty} N_t=-\infty\},
\end{equation}
which is a $Q_1$-null set by 
the Dubins-Schwarz 
theorem which
asserts that a continuous martingale is a time-changed Brownian motion.  
Therefore
\begin{equation}\label{barXhitting}
\bar X^\vep_{t+\vep}
\hbox{ hits $0$ or $1$ for some } t\le \bar T_1^\vep-\vep, t<\infty,
\ Q_1-a.s.,
\end{equation}
and therefore
\begin{align}\label{term2}
E\Bigl(\bar D^\vep(T_{\bar D}(0,x_0))x_0^{-1}1(\bar T_1^\vep=\infty)\Bigr)&\le Q_1(\bar T_1^\vep=\infty)\\
\nonumber&\le Q_1(\bar X_t^\vep\hbox{ hits }1\hbox{ for }t<\bar T^\vep_1)\\
&\le s_k(2b\vep)/s_k(1),
\end{align}
the last by \eqref{Q1hit}.

Since $\lim_{x\to 0+}s_k(x)/x=1$, there is an $\vep_0(k)>0$ and $x_0=x_0(k)\in(0,1)$, such that 
\begin{equation}\label{epsx0}\vep\le \vep_0\hbox{ implies }s_k(2b\vep)<3b\vep\hbox{ and }2b\vep<x_0\le s_k(1)/6.
\end{equation}
So for $\vep\le\vep_0$ and $x_0$ as above we may use \eqref{term1} and \eqref{term2} in \eqref{Q1Dhit} and conclude
\[Q_1
(\exists t\in[\vep,\bar T_1^\vep) : \bar D^\vep_t\ge x_0)
\ge(2b\vep/x_0)-(s_k(2b\vep)/s_k(1))> (b\vep)/x_0.\]
Virtually the same proof (it is actually simpler) works for $Q_0$.  Under $Q_i$, $\bar D^\vep_t=|\bar X^\vep_t-\bar Y^\vep_t|$ for $t\le \bar T^\vep_1=\inf\{t:\bar X^\vep_t\vee \bar Y^\vep_t=1\}$ and so we have proved for $x_0$ as above,
\begin{equation} \label{Qisep} Q_i
(\exists t\in[\vep,\bar T^\vep_1] : |\bar X^\vep_t-\bar Y^\vep_t|\ge x_0)
\ge\frac{b\vep}{x_0}\ \hbox{for }i=1,2\hbox{ and }0<\vep\le \vep_0,
\end{equation}
and (see \eqref{barXhitting}
for $i=1$)
\begin{equation}\label{barmaxhitting}
\bar X^\vep_t\vee\bar Y^\vep_t\hbox{ hits $0$ or $1$ for }t\le \bar T^\vep_1,\ t\hbox{ finite } Q_i-a.s.,\ i=1,2.
\end{equation}

Let 
\[\NN_1=\min\{j:(\bar X^\vep\vee\bar Y^\vep)(t+\bar T^\vep_j)\hbox{ hits $1$ for }t<\bar T_{j+1}^\vep-\bar T_j^\vep\},\]
and 
\[\NN_2=\min\{j:|\bar X^\vep-\bar Y^\vep|(t+\bar T^\vep_j)\hbox{ hits $x_0$ for }t<\bar T_{j+1}^\vep-\bar T_j^\vep\}.\]
Use \eqref{barexc1}, \eqref{barexc2} and \eqref{Qimaxhit} to see that
\begin{align*}
P(&\NN_1>n) \\
&=E\Bigl(1(\NN_1>n-1)P\bigl(\bar X^\vep\vee\bar Y^\vep((\bar T^\vep_{n-1}+\cdot)\wedge\bar T^\vep_n)\hbox{ doesn't hit }1\bigl|\cF_{\bar T^\vep_{n-1}}\Bigr)\Bigr)\\
&\ge P(\NN_1>n-1)\Bigl(1-\frac{s_k(2b\vep)}{s_k(1)}\Bigr).
\end{align*}
Therefore, if $p_1=\frac{s_k(2b\vep)}{s_k(1)}$, then
\begin{equation}\label{N1bnd} P(\NN_1>n)\ge (1-p_1)^{n+1}.
\end{equation}
Similar reasoning using \eqref{Qisep} in place of \eqref{Qimaxhit} shows that if $p_2=\frac{b\vep}{x_0}$, then for $\vep\le\vep_0$, 
\begin{equation}\label{N2bnd}
P(\NN_2>n)\le (1-p_2)^{n+1}.
\end{equation}
Note that \eqref{epsx0} shows that
\begin{equation}\label{pis}
\frac{p_2}{p_1}=\frac{b\vep}{s_k(2b\vep)}\frac{s_k(1)}{x_0}\ge\frac{1}{3}\frac{s_k(1)}{x_0}\ge 2.
\end{equation}
If $n=\lceil p_1^{-1}\rceil$ we get for $\vep\le\vep_0$
\begin{align*}
P(\NN_2<\NN_1)\ge P(\NN_1>n)-P(\NN_2>n)&\ge (1-p_1)^{n+1}-(1-p_2)^{n+1}\\
&\ge(1-p_1)^{n+1}-(1-2p_1)^{n+1}\\
&\ge\frac{1}{2}(e^{-1}-e^{-2}),
\end{align*}
where the last inequality holds by decreasing $\vep_0(k)$, if necessary.
If
\[\bar\bt_\vep=\inf\{t:\bar X^\vep_t\vee\bar Y^\vep_t\ge 1\},
\]
then the above bound implies that for $\vep\le \vep_0$,
\begin{equation}\label{sepprob1}
P(\sup_{t\le \bar \bt_\vep}|\bar X^\vep_t-\bar Y^\vep_t|\ge x_0)\ge \frac{1}{2}(e^{-1}-e^{-2}).
\end{equation}

Now let
\[\bt_\vep=\inf\{t:\langle X^\vep_t,1\rangle\vee\langle Y^\vep_t,1\rangle\ge 1\}.\]
Then \eqref{barXeqX} shows that
\[\hbox{if }\bt_\vep<V_k^\vep,\hbox{ then }\bar \bt_\vep=\bt_\vep\hbox{ and }(\bar X^\vep_t,\bar Y^\vep_t)=(\langle X^\vep_t,1\rangle,\langle Y^\vep_t,1\rangle)\hbox{ for all }t\le \bt_\vep,\]
and so by \eqref{sepprob1} for $\vep\le \vep_0$, 
\begin{equation}\label{sepprob2}
P(\sup_{t\le \bt_\vep}|\langle X^\vep_t,1\rangle-\langle Y^\vep_t,1\rangle|\ge x_0)\ge \frac{1}{2}(e^{-1}-e^{-2})-P(V^\vep_k\le \bt_\vep).
\end{equation}

Now recall we have $\vep_n\downarrow 0$ so that 
$(X^{\vep_n},Y^{\vep_n},W)\to(X,Y,W)$ weakly on 
$C(\R_+,(C^+_{rap})^2\times C_{tem})$, where $X$ and $Y$ are $C^+_{rap}$-valued 
solutions of \eqref{spde2}. Arguing as in \eqref{DS} and using Dubins-Schwarz, 
we see that
\begin{equation}\label{bigmass}\limsup_{t\to\infty}\langle X_t,1\rangle=\limsup_{t\to\infty}\langle Y_t,1\rangle=\infty\ \ a.s.
\end{equation}
Standard weak convergence arguments now show that $\{\bt_{\vep_n}\}$ are stochastically bounded.
Lemma~\ref{lem:unifHolder} therefore shows that we may choose a fixed $k$ sufficiently large so that
\[P(V_k^{\vep_n}\le \bt_{\vep_n})\le \frac{1}{4}(e^{-1}-e^{-2})\hbox{ for all }n.\]
Using this fixed $k$ throughout we see from \eqref{sepprob2} that for large enough $n$
\[P\Bigl(\sup_{t\le \bt_{\vep_n}}\bigl|\langle X^{\vep_n}_t,1\rangle-\langle Y^{\vep_n}_t,1\rangle\bigr|\ge x_0\Bigr)\ge \frac{1}{4}(e^{-1}-e^{-2}).\]
If $\bt'=\inf\{t:\langle 
X_t,1\rangle \lor \langle Y_t,1\rangle\ge 2
\}<\infty$ 
a.s.,
by \eqref{bigmass}, then the above implies 
\[P\Bigl(\sup_{t\le \bt'}\bigl|\langle X_t,1\rangle-\langle Y_t,1\rangle\bigr|\ge x_0/2\Bigr)\ge \frac{1}{4}(e^{-1}-e^{-2}),\]
and so $P(X\neq Y)\ge \frac{1}{4}(e^{-1}-e^{-2})$. 
\qed

\begin{lemma}\label{lem:unifHolder} For any $M\in\N$, $\lim_{k\to\infty}\sup_{0<\vep\le 1}P(V^\vep_k\le M)=0$.
\end{lemma}

\begin{proof}
The proof depends on a standard argument in the spirit of Kolmogorov's 
continuity lemma, so we will omit some details.  

Fix the time interval $[0,M]$.  Define
\begin{align*}V_k^\vep(X)=\inf\{s\ge 0:\ &\exists x,x'\in\R\hbox{ such that }\\
&|X^\vep(s,x)-X^\vep(s,x')|>k|x-x'|^{1/4}\}.\end{align*}
and $V_k^\vep(Y)$ likewise.  It suffices to prove Lemma \ref{lem:unifHolder} 
for $V_k^\vep$ replaced by $V_k^\vep(X)$ and $V_k^\vep(Y)$ and so clearly we only
need consider $V_k^\vep(X)$.  Recall that $\{X^{\vep_n}\}$ is tight in $C(\R_+,C_{rap}^+)$.  
So it suffices to choose a constant 
$K>0$ and prove the lemma for $X^\vep(t,x)\wedge (Ke^{-|x|})^{1/p}$ in place of $X^\vep$.
Considering the integral equation for $X^\vep$, and using the fact that 
$\psi\in C^1_c(\R)$
we see that it is enough to prove Lemma \ref{lem:unifHolder} with $X^\vep$ 
replaced by the stochastic convolution
\[
N^\vep(t,x) = \int_{0}^{t}p_{t-s}(x-y)\varphi^\vep(s,y)dW(s,y).
\]
Here one can use Lemma 6.2 of \cite{Shi94} to handle the drift terms. 
The term $\varphi^\vep(s,y)$ is a predictable random field satisfying 
\[
|\varphi^\vep(t,x)| \leq Ke^{-|x|}
\]
for all $t\in[0,M]$, $x\in\R$ almost surely.  Since our estimates are uniform in 
$\vep$, we will omit the superscript on 
$\varphi$ and $N$
from now on.  The constants below
may depend on $M$ and $K$.

Now we rely on some standard estimates which are easy to verify.  
We claim that there exist constants $q_0,K_0$ such that for 
$0\leq t\leq t+\delta\leq M$ and $x\in\mathbf{R}$, and for $\delta<1$,
\begin{eqnarray}
\label{Gbnds}\int_{0}^{\delta}\int_{\mathbf{R}}p_s^2(x-y)e^{-2p|y|}dyds 
  &\leq& K_0\delta^{\frac{1}{2}} e^{-q_0|x|} ,
   \\
\nonumber \int_{0}^{t}\int_{\mathbf{R}}[p_{t-s+\delta}(x-y)-p_{t-s}(x-y)]^2e^{-2p|y|}dyds 
  &\leq& K_0\delta^{\frac{1}{2}} e^{-q_0|x|}  ,
  \\
\nonumber\int_{0}^{t}\int_{\mathbf{R}}[p_{t-s}(x-y+\delta)-p_{t-s}(x-y)]^2
e^{-2|y|} 
dyds 
  &\leq& K_0\delta e^{-q_0|x|}  .
\end{eqnarray}

From these inequalities, it follows in a standard way that for some positive constants 
$q_1,C_0,C_1$, we have 
\begin{eqnarray}
\label{prob-ests}
P\left(|N(t+\delta,x)-N(t,x)|\geq\lambda\right) 
  &\leq& C_0\exp\left(-\frac{C_1\lambda^2}
       {\delta^{\frac{1}{2}}}\right)e^{-q_1|x|} 
,
  \\
P\left(|N(t,x)-N(t,x+\delta)|\geq\lambda\right) 
  &\leq& C_0
\exp\left(-\frac{C_1\lambda^2} 
        {\delta}\right)e^{-q_1|x|}.  \nonumber
\end{eqnarray}
For example, if we write
\[
\hat M_r= \int_{0}^{r}\int_{\mathbf{R}}[p_{t-s}(x-y+\delta)-p_{t-s}(x-y)]
      \varphi(s,y)W(dy,ds)
\]
then $\hat M_r$ is a continuous martingale and hence a time changed Brownian 
motion, with time scale
\begin{eqnarray*}
E(r)
&=& \int_{0}^{r}\int_{\mathbf{R}}[p_{t-s}(x-y+\delta)-p_{t-s}(x-y)]^2
      \varphi^2(s,y)dyds    \\
&\leq& \int_{0}^{r}\int_{\mathbf{R}}[p_{t-s}(x-y+\delta)-p_{t-s}(x-y)]^2
    K^2 
e^{-2|y|} 
dyds  .
\end{eqnarray*}
Thus, 
\begin{eqnarray*}
P\left(|N(t,x)-N(t,x+\delta)|\geq\lambda\right) &=& P(|\hat M_t|\geq\lambda) \\
&\leq& P\left(\sup_{0\leq s\leq E(t)}|B_s|\geq\lambda\right) 
\end{eqnarray*}
and then the reflection principle for Brownian motion  and the third inequality in \eqref{Gbnds} (to bound $E(t)$ for $t\le M$) gives the second 
inequality in (\ref{prob-ests}).

Now we outline a standard chaining argument, and for simplicity assume that 
$M=1$.  Let $\mathcal{G}_n$ be the grid of points 
\[
\mathcal{G}_n=\left\{\left(\frac{k}{2^{2n}},\frac{\ell}{2^{n}}\right):
   0\leq k\leq 2^{2n}, \ell\in\mathbf{Z}\right\}.
\]
The Borel-Cantelli lemma along with (\ref{prob-ests}) now implies that for 
large enough (random) $K_1$, if $n\ge K_1$ and $p_1,p_2$ are neighboring grid points 
in $\mathcal{G}_n$, then
\begin{equation}
\label{borel-cantelli}
|N(p_1)-N(p_2)|\leq 2^{-\frac{n}{4}}.
\end{equation}
Now suppose that $q_i=(t_i,x_i)$ with $|x_1-x_2|\leq1$, and that each point 
$q_i$ lies in some grid $\mathcal{G}_n$.  From the above, there is a path from 
$q_1$ to $q_2$ utilizing edges in grids $\mathcal{G}_{n'}$, with $n'\le n$, each edge in the path being 
a nearest neighbor edge in $\mathcal{G}_{n'}$, and with at most 8 edges 
from a given grid index $n'$.  Let $n_0$ be the least grid index used in this 
path.  We claim that for some constants $C>c>0$, such a path exists with $n_0$ satisfying
\begin{eqnarray*}
&&c2^{-2n_0}<|t_1-t_2|<C2^{-2n_0} 
, \\
&&c2^{-n_0}<|x_1-x_2|<C2^{-n_0}.
\end{eqnarray*}
Using the triangle inequality to 
sum
differences 
of $N(t,x)$ over edges of the path, we arrive at a geometric series, and 
conclude that
\begin{equation}\label{Kol1}
|N(q_1)-N(q_2)|\leq C_12^{-\frac{n_0}{4}}\hbox{ if }n_0\ge K_1.
\end{equation}
Although we have only proved the above for grid points,
such points are dense in $[0,T]\times\mathbf{R}$, and $N(t,x)$ has 
a continuous version because 
$X(t,x)$ 
 is continuous, and the drift contribution is
smooth.  Therefore it follows for all points in $[0,1]\times\R$. We have proved \eqref{Kol1} for $\Vert q_1-q_2\Vert\le C2^{-K_1}$ where $K_1$ is stochastically bounded uniformly in $\vep$. The required result follows. 
\end{proof}

\noindent{\bf Sketch of Proof of Theorem~\ref{Girseg}.} We carry out an excursion construction of an approximate solution $X^\vep$ to \eqref{signedeq} by starting the $i$th excursion at $(-1)^i\vep\psi$, 
and then run each independent excursion according to a fixed law of a $C^+_{rap}$-valued solution to 
\eqref{signedeq}
 with $X_0=\vep\psi$, if $i$ is even, and its negative if $i$ is odd, until the total mass hits $0$.  At this point a new excursion is started in the same manner. Theorem~\ref{RAL} is used to time change $X_t^\vep(1)$ into an approximate solution $Y^\vep(t)=X^\vep_{\tau^\vep_t}(1)$ of Girsanov's equation
\begin{equation}\label{girssde} dY_t=|Y_t|^{p'}dB_t,
\end{equation}
with $p<p'<1/2$ and $\frac{d\tau^{\vep}(t)}{dt}\le 1$. There will be an additional term $A^\vep(t)$ arising from all the excursion signed initial values up to time $t$ but it will converge to $0$ uniformly in $t$ due to the alternating nature of the sum.  We now proceed as in the excursion-based construction of non-zero solutions to Girsanov's sde \eqref{girssde} to show that one of the excursions of the approximate solutions will hit $\pm 1$ before time $T$ with probability close to $1$ as $T$ gets large, uniformly in $\vep$. Let $N^\vep$ be the number of excursions of $Y^\vep$ until one hits $\pm 1$ and let $N_\vep(T)$ be the number of excursions of $Y^\vep$ completed by time $T$.  $N^\vep$ is geometric with mean $\vep^{-1}$ by optional stopping.  Let 
$U_i(\vep)$ 
 be the time to completion of the $i$th excursion of $Y^\vep$.  Assuming $\sqrt T \vep^{-1}\in\N$, we have
\begin{align}
\nonumber P(\sup_{s\le T}|Y^\vep_s|\ge 1)&\ge P(N_\vep(T)\ge N_\vep)\\
\nonumber&\ge P(N_\vep(T)\ge \sqrt T \vep^{-1})-P(N_\vep>\sqrt T\vep^{-1})\\
\label{escbnd}&\ge P(U_{\sqrt T\vep^{-1}}(\vep)\le T)-(1-\vep)^{\sqrt T\vep^{-1}}.
\end{align}
 A key step now is to use diffusion theory to show that if $Y$ satisfies \eqref{girssde} (pathwise unique until it hits zero) then
\begin{equation}\label{domatt}P(Y_t>0\hbox{ for all }t\le T|Y_0=1))\sim 
cT^{-1/(2(1-p))}
 \hbox{ as }T\to\infty.\end{equation}
If $U_i(1)$ is the time of completion of the $i$th excursion of $Y$ where the excursions now start at $\pm 1$, then scaling shows that 
\[P(U_{\sqrt T\vep^{-1}}(\vep)\le T)=P(U_{\sqrt T\vep^{-1}}(1)\le \vep^{-(2-2p)}T).\]
\eqref{domatt} shows that $U_{\sqrt T\vep^{-1}}(1)/(\sqrt T \vep^{-1})^{2(1-p)}$ converges weakly as $\vep\downarrow 0$ to a stable subordinator of index $\alpha=(2(1-p))^{-1}$  and so for any $\eta>0$ we may choose $T$ large enough so that for small enough $\vep$ (by \eqref{escbnd}) we have 
\begin{align*}
P(\sup_{s\le T}|Y^\vep_s|\ge 1)&\ge P(U_{\sqrt T\vep^{-1}}(\vep)\le T)-(1-\vep)^{\sqrt T\vep^{-1}}\\
&\ge P
\left(
 \frac{U_{\sqrt T\vep^{-1}}(1)}{(\sqrt T \vep^{-1})^{2(1-p)}}\le T^p
\right)
 -e^{-\sqrt T}\ge 1-\eta.
\end{align*}
The fact that $(\tau^\vep)'(t)\le 1$ allows us to conclude that  with probability at least $1-\eta$, uniformly in $\vep$, the total mass of our approximate solution $X_t^\vep$ will hit $\pm1$ for some $t\le T$.  By taking a weak limit point of the $X^\vep$ we obtain the required non-zero solution to \eqref{signedeq}.\qed

\section{Pathwise Non-uniqueness and Uniqueness in Law for an SDE.}
\label{sec:sde}

The stochastic differential equation corresponding to \eqref{spde2} would be 
\begin{equation}\label{sde}
X_t=X_0+bt+\int_0^t(X_s)^p\,dB_s,\ \ X_t\ge 0\ \forall t\ge 0\ \ a.s.
\end{equation}
Here $b>0$, $0<p<1/2$, $B$ is a standard $(\cF_t)$-Brownian motion on $(\Omega,\cF,\cF_t,P)$ and $X_0$ is $\cF_0$-measurable.  A much simpler argument than that used to prove pathwise non-uniqueness 
in
 \eqref{spde2} allows one to establish pathwise non-uniqueness in \eqref{sde}. One only needs to apply the idea behind construction of $(\bar X_t^\vep,\bar Y_t^\vep)$ for $t\ge V^\vep_k$.  In any case  the result is undoubtedly known, given the well-known Girsanov examples (see, e.g., Section V.26 in \cite{RW}) and so we omit the proof.

\begin{theorem} \label{thm:sdenpu} There is a filtered probability space $(\Omega,\cF,\cF_t,P)$ carrying a standard $\cF_t$-Brownian motion and two solutions, $X^1$ and $X^2$, to \eqref{sde} with $X_0^1=X_0^2=X_0=0$ such that $P(X^1\neq X^2)>0$.
\end{theorem}

Weak existence of solutions to \eqref{sde} for a given initial law may be 
constructed through approximation by Lipschitz coefficients.  
This
is in fact how Shiga \cite{Shi94} constructed solutions to (1.1) and hence
is the method used in Theorem 1.
As we were not able to verify 
whether or not uniqueness in law holds in \eqref{spde2} it is perhaps 
interesting that it does hold in \eqref{sde}.  That is, the law of $X$ is 
uniquely determined by the law of $X_0$.  We have not been able to find this 
result in the literature and since the solutions to \eqref{sde} turn out to 
have a particular sticky boundary condition at $0$ which was not immediately 
obvious to us, we include the elementary proof here. 

\begin{theorem}\label{thm:sdeul} Any solution to \eqref{sde} is the diffusion on $[0,\infty)$ with scale function
\begin{equation}\label{scale} s(x)=\int_0^x\exp\Bigl\{\frac{-2b|y|^{1-2p}}{1-2p}\Bigr\}\,dy
\end{equation}
(with inverse function $s^{-1}$ on $[0,s(\infty)$),
speed measure
\begin{equation}\label{speed}
m(dx)=\frac{dx}{s'(s^{-1}(x))^2s^{-1}(x)^{2p}}+b^{-1}\delta_0(dx)\hbox{ on }[0,s(\infty)),
\end{equation}
and 
starting
with the law of $X_0$.  
In particular if $T_0=\inf\{t:X_t=0\}$, then 
\begin{equation}\label{sticky0} T_0<\infty\hbox{ implies }\int_{T_0}^{T_0+\vep}1(X_s=0)\,ds>0\ \ \forall\vep>0\ a.s.,
\end{equation}
and solutions to \eqref{sde} are unique in law. 
\end{theorem}
\begin{proof} The last statement is immediate from the first assertion.  

To prove $X$ is the diffusion described above, by conditioning on $X_0$ we may assume $X_0=x_0$ is constant.  We will show directly that $X$ is the appropriate scale and time change of a reflecting Brownian motion.  Note that 
$s$
 is strictly increasing on $[0,\infty)$ (in fact on the entire real line) so that $s^{-1}$ is well-defined.  Note also that
\[s'(x)=\exp\Bigl\{\frac{-2b|x|^{1-2p}}{1-2p}\Bigr\}\hbox{ is of bounded variation and continuous},\]
and 
\begin{equation}\label{s"form}
s''(x)=\begin{cases}-s'(x)2bx^{-2p} &\text{ if }x>0
,
\\
s'(x)2b|x|^{-2p} &\text{ if }x<0.
\end{cases}
\end{equation}
If $L^X_t(x)$ is the semimartingale local time of $X$, Meyer's generalized It\^o formula (see Section IV.45 of \cite{RW}) shows that
\begin{equation}\label{GIto}
Y_t\equiv s(X_t)=Y_0+\int_0^t s'(X_u)X_u^p\,dB_u+b\int_0^ts'(X_u)\,du+\frac{1}{2}\int L^X_t(x)ds'(x).\end{equation}
Since $s'$ is continuous at $0$, 
\begin{align}
\label{LTform}\frac{1}{2}\int L_t^X(x)ds'(x)&=\frac{1}{2}\int 1(x>0)L^X_t(x)ds'(x)\\
\nonumber&=\frac{1}{2}\int1(x>0)L^X_t(x)s''(x)dx\\
\nonumber&=-b\int_0^ts'(X_u)X_u^{-2p}X_u^{2p}1(X_u>0)\,du\ \ (\hbox{by }\eqref{s"form})\\
\nonumber&=-b\int_0^ts'(X_u)1(X_u>0)\,du.
\end{align}
So \eqref{GIto} and $s'(0)=1$ imply
\begin{equation}\label{Yeq}
Y_t=Y_0+\int_0^t s'(X_u)X_u^p\,dB_u+b\int_0^t1(X_u=0)\,du.
\end{equation}
Define $U=\int_0^\infty s'(X_u)^2X_u^{2p}\,du$ and
a random time change $\alpha:[0,U)\to[0,\infty)$ by
\begin{equation}\label{alpha}
\int_0^{\alpha (t)} s'(X_u)^2X_u^{2p}\,du=t.
\end{equation}
Clearly $\alpha$ is strictly increasing and is also continuous since $X$ cannot be $0$ on any interval.  
If $R(t)=Y(\alpha(t))$ for $t<U$ we now show that $R$ is a reflecting Brownian motion on $[0,\infty)$, starting at $Y_0$, where we extend the definition for $t\ge U$ by appending 
a
 conditionally independent
reflecting Brownian motion starting at the appropriate point. In what follows we may assume $t<U$ as the values of 
$R(t)$
 for $t\ge U$ will not be relevant. We have from \eqref{Yeq}
\[R(t)=\beta_t+b\int_0^{\alpha(t)}1(X_u=0)\,du\equiv\beta_t+A_t,\]
where $\langle\beta\rangle_t=t$ and so $\beta$ is a Brownian motion starting at $Y_0$.  $A$ is continuous non-decreasing and supported by $\{t:X(\alpha(t))=0\}=\{t:R(t)=0\}$.  By uniqueness of the Skorokhod problem (see Section V.6 in \cite{RW}) $R$ is a reflecting Brownian motion and $A_t=L_t^R(0)$, that is,
\begin{equation}\label{LTR} b\int_0^{\alpha(t)} 1(X_u=0)\,du=L^R_t(0).
\end{equation}
Let $\alpha^{-1}:[0,\infty)\to [0,U)$ denote the inverse function to $\alpha$. Now differentiate \eqref{alpha} to see that 
\begin{equation}\label{alpha'}
\hbox{if }X_u>0, \hbox{ then }(\alpha^{-1})'(u)=s'(X_u)^2X_u^{2p}.
\end{equation}
We may use \eqref{LTR} and \eqref{alpha'} to conclude that 
\begin{align*}
t&=\int_0^t 1(X_u>0)\,du+\int_0^t 1(X_u=0)\,du\\
&=\int_0^t\frac{1(X_u>0)}{s'(X_u)^2X_u^{2p}}\,d(\alpha^{-1}(u))+b^{-1}L^R_{\alpha^{-1}(t)}(0)\\
&=\int_0^{\alpha^{-1}(t)}\frac{1(R
(v)
>0)}{s'(s^{-1}(
R(v)
))^2s^{-1}(R(v))^{2p}}\,dv+b^{-1}L^R_{\alpha^{-1}(t)}(0),
\end{align*}
where we have set $u=\alpha(v)$ in the last.  Therefore if $m$ is as in \eqref{speed}, then
\[t=\int_{[0,\infty)}L^R(\alpha^{-1}(t),x)dm(x),\]
and
\[X(t)=s^{-1}(R(\alpha^{-1}(t)).\]
This identifies $X$ as the diffusion on $[0,\infty)$ with the given scale function and speed measure.
\end{proof}

\noindent{\bf Remarks.} (1) One can also argue in the opposite direction.
That is, given a diffusion $X$ with speed measure and scale function as above and a given initial law on $[0,\infty)$, one 
can build a Brownian motion $B$, perhaps on an enlarged probability space,  so that $X$ satisfies \eqref{sde}, giving us an alternative weak existence proof.
\medskip

\noindent (2) One can construct solutions as weak limits of difference equations or equivalently as standard parts of an infinitesimal difference equation.  Here one cuts off the martingale part when the solution overshoots into the negative half-line and lets the positive drift with slope $b$ bring it back to $\R_+$.  The smaller the $b$ the longer it takes to become positive, the more time the solution
will spend at zero and so the larger the atom of the speed measure at $0$. A short calculation shows
that at $p=1/2$ the overshoot reduces to $\Delta t$ (the time step in the difference equation) and so
there is 
no
time spent at $0$ in the limit. (See Section V.48 of \cite{RW} for the standard analysis.)
\medskip

\noindent (3) It would appear that \eqref{sde} is not a particular effective 
tool to study diffusions with drift $b$ on the positive half-line.  By just 
extending the equation to $[0,\infty)$ we inadvertently pick out a particular 
case of Feller's possible boundary behaviors at $0$ among all diffusions 
satisfying \eqref{sde} on $(0,\infty)$.  (This is certainly not a novel 
observation---see the comments in Section V.48 in \cite{RW}.) Presumably things 
can only get worse for the stochastic pde \eqref{spde2}.  In the next section 
we scratch the surface of this issue and show that all solutions to this 
stochastic pde spend positive time in the (infinite-dimensional) zero state. 
    
\section{Proof of Theorem~\ref{sticky}}\label{sec:sticky}

Let $X$ be a solution of \eqref{spde2} and define
\[V_k=\inf\{s
:\exists x',x\hbox{ such that }|
X(s,x)
-X(s,x')|>k|x-x'|^{1/4}\}.\]
As in Lemma~\ref{lem:unifHolder} (but as there is no $\vep$ it is a bit easier), $\lim_kV_k=\infty$ a.s.  As in Section~\ref{sec:proofnonpu}, we set $p'=\frac{p+2}{5}$ and $K(k)=K_{\ref{RAL}}(1/4,k)$. If we define
\[R(u)=\begin{cases}\frac {\int X(u,x)^{2p}\,dx}{K(k)\langle X(u),1\rangle^{2p'}}&\text{ if }\langle X(u),1\rangle>0\text{ and }u\le V_k,\\
\qquad 1&\text{ otherwise,}
\end{cases}\]
then by Theorem~\ref{RAL},
\[R(u)\ge 1\hbox{ for all }u\ge 0.\]
Introduce a random time change $\tau$ given by
\begin{equation}\label{tau} 
\int_0^{\tau(t)} R(u)1(\langle X_u,1\rangle>0)+1(\langle X_u,1\rangle=0)\,du=t.
\end{equation}
Clearly $\tau$ is strictly increasing, continuous and well-defined for all $t\ge 0$.  
Differentiate \eqref{tau} to see that 
\[\tau'(t) R_{\tau(t)}1(\langle X_{\tau(t)},1\rangle>0)+\tau'(t)1(\langle X_{\tau(t)},1\rangle=0)=1\quad\hbox{ for a.a. }t\ge0\]
(a.a. is with respect to Lebesgue measure), and therefore
\begin{equation}\label{tau'}\tau'(t)=R^{-1}_{\tau(t)}1(\langle X_{\tau(t)},1\rangle>0)+1(\langle X_{\tau(t)},1\rangle=0)\le 1\quad\hbox{ for a.a. }t\ge 0.
\end{equation}
Now let
\[Y(t)=\langle X(\tau(t)),1\rangle=b\tau(t)+M(t),\]
where $M$ is a continuous local martingale satisfying
\begin{align*}
\langle M\rangle_t&=\int_0^{\tau(t)}\int X(u,x)^{2p}\,dx\,du\\
&=\int_0^t\int X(\tau(r),x)^{2p}\,dx\Bigl[R(\tau(r))^{-1}1(Y(r)>0)+1(Y(r)=0)\Bigr]dr\\
&=\int_0^tK(k)Y_r^{2p'}\,dr\qquad\hbox{for }\tau(t)\le V_k.
\end{align*}
We have used \eqref{tau'} in the second line.  If $T_k=\tau^{-1}(V_k)$ (a stopping time w.r.t the time-changed filtration),  we may therefore assume there is a Brownian motion $B$ so that
\[Y(t\wedge T_k)=b\tau(t\wedge T_k)+\int_0^{t\wedge T_k}\sqrt{K(k)}Y^{p'}_r\,dB_r.\]
If $b'=K(k)^{1/(2(p'-1))}b$, then $\hat Y(t)=K(k)^{1/(2(p'-1))} Y(t)$ 
satisfies
\[\hat Y(t\wedge T_k)=b'\tau(t\wedge T_k)+\int_0^{t\wedge T_k}\hat Y_r^{p'}\,dB_r.\]
If $s(x)=\int_0^x\exp\Bigl\{\frac{-2b'|y|^{1-2p'}}{1-2p'}\Bigl\}\,dy$, then an application of Meyer's generalized It\^o's formula shows that if $Z(t)=s(\hat Y(t))$, then for $t\le T_k$,
\[Z(t)=\int_0^t s'(\hat Y_r)\hat Y_r^{p'}\,dB_r+b'\int_0^ts'(\hat Y_r)\tau'(r)\,dr+\frac{1}{2}\int L^{\hat Y}_t(x)ds'(x).\]
Here, as before, $L^{\hat Y}$ is the semimartingale local time of $\hat Y$. Now argue as in \eqref{LTform} to see that for $t\le T_k$,
\begin{align*}
\frac{1}{2}\int L_t^{\hat Y}(x)ds'(x)=-b'\int_0^t s'(\hat Y_r) 1(\hat Y_r>0)\,dr.
\end{align*} 
Therefore if $N(t)=\int_0^t s'(\hat Y_r)\hat Y_r^{p'}\,dB_r$ and 
\[A(t)=b'\int_0^t s'(\hat Y_r)(1-\tau'(r))1(\hat Y_r>0)\,dr,\]
then for $t\le T_k$,
\begin{align}
\nonumber Z(t)&=N(t)-A(t)+b'\int_0^ts'(0)\tau'(r)1(\hat Y(r)=0)\,dr\\
\label{Zform}&=N(t)-A(t)+b'\int_0^t1(Y(r)=0)\,dr,
\end{align}
where we used $s'(0)=1$ and \eqref{tau'} in the last line. $A$ is a non-decreasing continuous process by \eqref{tau'}, $N$ is a continuous local martingale, and $N(0)=A(0)=0$.  Fix $k$ and assume $V_k>0$, and so $T_k>0$ because $T_k\ge V_k$.  If 
\[T_+=\inf\Bigl\{t:\int_0^{t\wedge T_k}1(Y(r)=0)\,dr>0\Bigr\},\]
then by \eqref{Zform}, $Z(t\wedge T_+)$ is a continuous non-negative local supermartingale starting at $0$ and so
is identically zero.  This means $Z(r)=0$ for $r\le T_+$ and so the same holds for $Y(r)$, which by the definition of $T_+$ and assumption that $T_k>0$ implies that $T_+=0$ a.s.  Since $V_k\uparrow\infty$ a.s. we have shown that w.p. 1 
\[\int_0^t1(\langle X(\tau(r)),1\rangle =0)\,dr=\int_0^t1(Y(r)=0)\,dr>0\ \forall t>0.\]
Setting $\tau(r)=u$ and using \eqref{tau'} again (to show $\tau'(r)=1$ on $\{\langle X(\tau(r)),1\rangle =0\}$ for a.a. $r$), we see that
the above implies
\[\int_0^t1(\langle X_u,1\rangle =0)\,du>0\ \ \forall t>0\ \ a.s.\]
The proof is complete.  \qed

\medskip
\noindent{\bf Acknowledgement.} It is a pleasure to thank Martin Barlow for a series of helpful conversations
on this work.  
We are grateful to the referee for very helpful suggestions.

\end{document}